\documentclass[reqno]{amsart}
\usepackage{amsmath, amsfonts, amssymb, amsthm, amscd, cancel, enumerate, rotating, comment, appendix,graphics,pgf,tikz,hyperref}
\usepackage[capitalise, noabbrev]{cleveref}
\usetikzlibrary{arrows,positioning, shapes.geometric,shadings}
\usepackage{times}

\usepackage{mathrsfs}

\usepackage{amsfonts,amssymb,amscd,amsmath,enumerate,verbatim,calc,url}
\usepackage{amsthm, psfrag,latexsym,epsfig,mdwlist,graphicx,comment, cancel}
\usepackage{times}  
\theoremstyle{plain}

\newtheorem{theorem}{Theorem}[section]

\newtheorem{proposition}[theorem]{Proposition}

\theoremstyle{definition}
\newtheorem{definition}[theorem]{Definition}

\newtheorem{remark}[theorem]{Remark}
\newtheorem{example}[theorem]{Example}

\newcommand{\st}{\ : \ }
\newcommand{\Taylor}{\mathrm{Taylor}}

\newcommand{\m}{\mathbf{m}}

\newcommand{\LCM}{\mathrm{LCM}}
\newcommand{\D}{\Delta}
\newcommand{\LL}{\mathbb{L}}
\newcommand{\tuple}[1]{\langle #1 \rangle}
\newcommand{\n}{\mathfrak{n}}

\newcommand{\sfk}{\mathsf k}

\begin{document}
\bibliographystyle{amsplain}

\author[S.~Cooper]{Susan M. Cooper}
\address{Department of Mathematics\\
University of Manitoba\\
520 Machray Hall\\
186 Dysart Road\\
Winnipeg, MB\\
Canada R3T 2N2}
\email{susan.cooper@umanitoba.ca}

\author[S.~El Khoury]{Sabine El Khoury}
\address{Department of Mathematics,
American University of Beirut,
Bliss Hall 315, P.O. Box 11-0236,  Beirut 1107-2020,
Lebanon}
\email{se24@aub.edu.lb}

\author[S.~Faridi]{Sara Faridi}
\address{Department of Mathematics \& Statistics\\
Dalhousie University\\
6316 Coburg Rd.\\
PO BOX 15000\\
Halifax, NS\\
Canada B3H 4R2}
\email{faridi@dal.ca}

\author[S~Mayes-Tang]{Sarah Mayes-Tang}
\address{Department of Mathematics\\
University of Toronto\\
40 St. George Street, Room 6290\\
Toronto, ON \\
Canada M5S 2E4}
\email{smt@math.toronto.edu}

\author[S.~Morey]{Susan Morey}
\address{Department of Mathematics\\
Texas State University\\
601 University Dr.\\
San Marcos, TX 78666\\U.S.A.}
\email{morey@txstate.edu}

\author[L.~M.~\c{S}ega]{Liana M.~\c{S}ega}
\address{Liana M.~\c{S}ega\\ Department of Mathematics and Statistics\\
   University of Missouri\\ \linebreak Kansas City\\ MO 64110\\ U.S.A.}
     \email{segal@umkc.edu}

\author[S.~Spiroff]{Sandra Spiroff }
\address{Department of Mathematics,
University of Mississippi,
Hume Hall 335, P.O. Box 1848, University, MS 38677
USA}
\email{spiroff@olemiss.edu}



 \title{Simplicial resolutions for the second power of square-free
   monomial ideals}

 \begin{abstract}  Given a square-free monomial ideal $I$, we define a
simplicial complex labeled by the generators of $I^2$ which supports a
free resolution of $I^2$.  As a consequence, we obtain (sharp) upper bounds
on the Betti numbers of the second power of any square-free monomial
ideal.
 \end{abstract} 
 
 \maketitle
%%%%%%%%%%%%%%%%%%%%%%%%%%%%%%%%%%%%%%%%%%%%%%%%%%%%%%%%%%%%%%%%%

%%%%%%%%%%%%%%%%%%%%%%%%%%%%%%%%%%%%%%%%%%%%%%%%%%%%%%%%%%%%%%%%%
\section{Introduction}
%%%%%%%%%%%%%%%%%%%%%%%%%%%%%%%%%%%%%%%%%%%%%%%%%%%%%%%%%%%%%%%%%

  The question of finding, or even effectively bounding, the Betti
  numbers of an ideal in a commutative ring is a difficult one. Even
  more complicated is using the structure of an ideal $I$ to find
  information about the Betti numbers of its powers $I^r$: predicting
  something as basic as the minimal number of generators of  $I^r$ is
  a difficult problem.

Taylor's thesis~\cite{T} described a free resolution of any ideal minimally
generated by $q$ monomials using the simplicial chain complex of a
simplex with $q$ vertices. Taylor's  construction, though often far from
minimal, produces a resolution of \emph{every} monomial ideal $I$. It gives
upper bounds ${q \choose {i+1} } \geq \beta_{i}(I)$ for the Betti
numbers of $I$ where $q \choose i+1$ is the number of $i$-faces of a
$q$-simplex.  If $I$ is generated by $q$ monomials and $r$ is a
positive integer, then the number of generators of $I^r$ generally grows
exponentially and as a result, so do the bounds on the Betti numbers
of $I^r$ given by Taylor's resolution.

In this paper, we focus on the case where $r=2$ and $I$ is a
square-free monomial ideal with $q$ generators. In this case, we know
that $I^2$ can be generated by at most ${q+1}\choose{2}$ monomials,
and hence has a Taylor resolution supported on a simplex with
at most ${q+1}\choose{2}$ vertices.  The question that
we address in this paper is: can we find a subcomplex of this simplex
whose simplicial chain complex yields a free resolution of $I^2$?  Such a
resolution would be closer to minimal than the Taylor resolution.

We answer this question by constructing a simplicial complex on
${q+1}\choose{2}$ vertices which we call $\LL^2_q$ in honor of the
Lyubeznik resolution~\cite{L} which was our inspiration.  While
$\LL^2_q$ has the same number of vertices as the
${q+1}\choose{2}$-simplex, it is significantly smaller because
it has far fewer faces.  For a given square-free monomial ideal
$I$, we can use further deletions of $\LL^2_q$ specific to the
generators of $I$ to show that $\LL^2_q$ has an induced subcomplex $\LL^2(I)$
which supports a free resolution of $I^2$.  As a result, we find (sharp) upper
bounds on the Betti numbers of the second power of any square-free
monomial ideal.  These bounds are often significantly smaller than
the bounds provided by the Taylor resolution (see \cref{s:bounds}).

  \cref{sec:background} lays out the notation and terminology used in
  the paper including the construction of simplicial
    resolutions. In \cref{sec:main}, we describe the complexes
  $\LL^2_q$ (\cref{d:L2}) and $\LL^2(I)$ (\cref{d:L2I}) and prove that
  $\LL^2(I)$ supports a free resolution of $I^2$ when $I$ is a
  square-free monomial ideal
  (\cref{t:L2-supports-new}). \cref{s:bounds} provides results on the
  bounds on the Betti numbers that follow from the main results.

  This paper is part of a larger project~\cite{CEFMMSS} to study
resolutions of powers of monomial ideals, which the authors started
during the 2019 Banff workshop ``Women in Commutative Algebra''.

%%%%%%%%%%%%%%%%%%%%%%%%%%%%%%%%%%%%%%%%%%%%%%%%%%%%%%%%%%%%%%%%%
\section{Background}
\label{sec:background}
%%%%%%%%%%%%%%%%%%%%%%%%%%%%%%%%%%%%%%%%%%%%%%%%%%%%%%%%%%%%%%%%%

Throughout this paper we let $S=\sfk[x_1,\ldots,x_n]$ be a polynomial
ring over a field $\sfk$.  In this section we briefly recall some
necessary background about simplicial complexes. 

A {\bf simplicial complex}
  $\D$ over a {\bf vertex set } $V$ is a set of subsets of $V$
  such that if $F \in \D$ and $G \subseteq F$ then $G \in \D$. An element $\sigma$ of
  $\D$ is called a {\bf face} and the maximal faces under
  inclusion are called {\bf facets}. A simplicial complex can be
  uniquely determined by its facets, and we use the
  notation $$\D=\tuple{F_0,\ldots,F_q}$$ to describe a simplicial
  complex whose facets are $F_0,\ldots,F_q$.

The {\bf dimension} of a face $F$ in $\D$ is $\dim(F)=|F|-1$, and the
dimension of $\D$ is the maximum of the dimensions of its faces.

A simplicial complex with one facet is called a {\bf simplex}.

  If $W \subseteq V$, the subcomplex
  $$\D_W=\{ \sigma \in \D \mid \sigma \subseteq W\}$$
  is called the {\bf induced subcomplex} of $\D$ on
  $W$.

  If $\Delta$ is
  a simplicial complex with vertex $v$, then when we {\bf delete} $v$
  from $\Delta$ we obtain the simplicial complex
  $$\Delta \setminus \{v \}= \{ \sigma \in \Delta \mid v \notin \sigma \}.$$

  A facet $F$ of $\D$ is said to be a {\bf leaf} if it is the only 
  facet of $\Delta$, or there is a different facet $G$ of $\Delta$, called a {\bf joint}, such that
  $$F\cap H \subseteq G$$ for all facets $H \neq F$. The joint  $G$ in
  this definition is not
  unique~(\cite{F02}).  
  A simplicial complex
  $\D$ is a {\bf quasi-forest} if the facets of $\D$ can be ordered as
  $F_0,\ldots, F_q$ such that for $i=0,\ldots,q$, the facet $F_i$ is a
  leaf of the simplicial complex $\tuple{F_0,\ldots,F_i}$. A connected
  quasi-forest is called a {\bf quasi-tree}~(\cite{Z}).

  \begin{example}\label{e:running-0} The simplicial complex below
    is a quasi-tree, with leaf order: $F_0,F_1,F_2,F_3$, meaning that
    each $F_i$ is a leaf of $\langle F_0,\ldots,F_i\rangle$. Note that in this case, the joint of $F_i$ is
    $F_0$ for every $i\geq 1$.
    
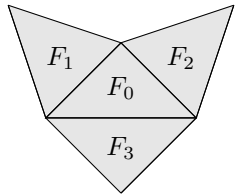
\[ 
\begin{tikzpicture}
\tikzstyle{point}=[inner sep=0pt] \coordinate (a) at (0,1);
\coordinate (b) at (-1,0); \coordinate (c) at (1,0); \coordinate (d)
at (-1.5,1.5); \coordinate (e) at (1.5,1.5); \coordinate (f) at
(0,-1); \draw [fill=gray!20](a.center) -- (b.center) -- (c.center);
\draw [fill=gray!20](a.center) -- (b.center) -- (d.center); \draw
      [fill=gray!20](a.center) -- (c.center) -- (e.center); \draw
      [fill=gray!20](b.center) -- (c.center) -- (f.center); \draw
      (a.center) -- (b.center); \draw (a.center) -- (c.center); \draw
      (a.center) -- (d.center); \draw (a.center) -- (e.center); \draw
      (b.center) -- (c.center); \draw (b.center) -- (d.center); \draw
      (b.center) -- (f.center); \draw (c.center) -- (e.center); \draw
      (c.center) -- (f.center);
      \node [label=$F_0$] at (0,0) {};
      \node [label=$F_1$] at (-.8,.4) {};
      \node [label=$F_2$] at (0.8,.4) {};
      \node [label=$F_3$] at (0,-.8) {};
\end{tikzpicture} \]

This complex is in fact $\LL^2_3$ as we will see later in
\cref{e:L2-picture}.

\end{example}

If $I$ is minimally generated by monomials
$m_1,\ldots,m_q$ in $S$, a {\bf minimal free resolution} of $I$ is a (unique up to isomorphism) exact sequence of free $S$-modules
$$0 \to S^{\beta_p} \to S^{\beta_{p-1}} \to \cdots \to S^{\beta_1} \to
S^{\beta_0} \to I \to 0$$ where $p \in \mathbb{N}$, $\beta_0=q$ and
for each $ i \in \{1,\ldots,p\}$, $\beta_i$ is the smallest possible
rank of a free module in the $i$-th spot of any free resolution of $I$. The
$\beta_i$, called the {\bf Betti numbers} of $I$, are invariants of
the ideal $I$.

Finding ways to describe a free resolution of a given ideal is an open
and active area of research. For monomial ideals, combinatorics plays
a big role.  In her thesis in the 1960's, Diana Taylor introduced a
method of labeling the faces of a simplex $\Delta$ with monomials, and
then used this labeling to turn the simplicial chain complex of
$\Delta$ into a free resolution of a monomial ideal. This technique
has been generalized to other simplicial complexes by Bayer and
Sturmfels~\cite{BS}, among others.

More precisely, if $I$ is minimally generated by monomials
$m_1,\ldots,m_q$ and $\Delta$ is a simplicial complex on $q$ vertices
$v_1,\ldots,v_q$, we label each vertex $v_i$ with the monomial $m_i$,
and we label each face of $\Delta$ with the least common multiple of
the labels of its vertices. Then, if the labeling of $\Delta$
satisfies certain properties, the simplicial chain complex of $\Delta$
can be ``homogenized'' using the monomial labels on the faces to give
a free resolution of $I$. In this case, we say that $\Delta$ {\bf
  supports a free resolution} of $I$ and the resulting free resolution
is called a {\bf simplicial resolution} of $I$.  Peeva's book~\cite{P}
details this method for simplicial as well as other topological
resolutions.

\begin{example}\label{e:running-1}
 Let $I=(x^2,y^2,z^2,xy,xz,yz)$.  In the picture below, we label the
 simplicial complex $\Delta$ in \cref{e:running-0} using the generators of
 $I$. To make the picture less busy, we have included the labels of
 the vertices and the facets only.
  
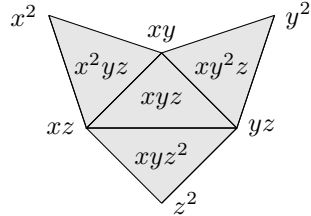
\[ 
\begin{tikzpicture}
\tikzstyle{point}=[inner sep=0pt]
\node (a)[point,label=above:$xy$] at (0,1) {};
\node (b)[point,label=left:$xz$] at (-1,0) {};
\node (c)[point,label=right:$yz$] at (1,0) {};
\node (d)[point,label=left:$x^2$] at (-1.5,1.5) {};
\node (e)[point,label=right:$y^2$] at (1.5,1.5) {};
\node (f)[point,label=right:$z^2$] at (0,-1) {};
\draw [fill=gray!20](a.center) -- (b.center) -- (c.center);
\draw [fill=gray!20](a.center) -- (b.center) -- (d.center);
\draw [fill=gray!20](a.center) -- (c.center) -- (e.center);
\draw [fill=gray!20](b.center) -- (c.center) -- (f.center);
\draw (a.center) -- (b.center);
\draw (a.center) -- (c.center);
\draw (a.center) -- (d.center);
\draw (a.center) -- (e.center);
\draw (b.center) -- (c.center);
\draw (b.center) -- (d.center);
\draw (b.center) -- (f.center);
\draw (c.center) -- (e.center);
\draw (c.center) -- (f.center);
\node [label=$xyz$] at (0,0) {};
      \node [label=$x^2yz$] at (-.8,.4) {};
      \node [label=$xy^2z$] at (0.8,.4) {};
      \node [label=$xyz^2$] at (0,-.8) {};
\end{tikzpicture} \]

Our main result \cref{t:L2-supports-new} will prove that, indeed, $\Delta$ does
support a free resolution of $I$. This in particular implies that $\beta_i(I)$ is bounded above by the number of $i$-faces of $\Delta$, which is the rank of the $i$-th chain group of $\Delta$. That is,
$$\beta_0(I) \leq 6, \quad \beta_1(I) \leq 9, \quad \beta_2(I) \leq 4,
\quad \beta_i(I) =0 \mbox{ if } i>2.$$ We calculate, using Macaulay2~\cite{M2}, that the actual Betti numbers of $I$ are:
$$\beta_0(I) = 6, \quad \beta_1(I) =8, \quad \beta_2(I) =3,
\quad \beta_i(I) =0 \mbox{ if } i>2.$$ 
\end{example}

A major question in the theory of combinatorial resolutions is to
determine whether a given simplicial complex supports a free resolution
of a given monomial ideal. Taylor proved
that a simplex with $q$ vertices \emph{always} supports a free
resolution of an ideal with $q$ generators, or in other words, \emph{every} monomial ideal has a 
Taylor resolution. As a result ${q}\choose{i+1}$ (the
number of $i$-faces of a simplex with $q$ vertices) is an upper bound
for $\beta_i(I)$ if $I$ is \emph{any} monomial ideal with $q$
generators.  We denote the $q$-simplex labeled with the
$q$ generators of $I$ by  $\Taylor(I)$.

Taylor's resolution is usually
far from minimal. However, if $I$ is a monomial ideal with a free
resolution supported on a (labeled) simplicial complex $\Delta$, then
$\Delta$ has to be a subcomplex of $\Taylor(I)$. As a result, the
question of finding smaller simplicial resolutions of $I$ turns into a
question of finding smaller subcomplexes of $\Taylor(I)$ which support
a resolution of $I$.

One of the best known tools to identify such subcomplexes of the
Taylor complex is due to Bayer, Peeva, and Sturmfels~\cite{BPS}, and reduces the problem to checking acyclicity of induced subcomplexes.  This
criterion was adapted in~\cite{F14} to the class of simplicial trees,
and then in ~\cite{CEFMMSS} to quasi-trees. \cref{t:res-by-quasitree}
is this latter adaptation, and will be used in the rest of the paper.

If $\D$ is a subcomplex of $\Taylor(I)$ and $\m$ is a monomial in
$S$, let $\D_\m$ be the subcomplex of $\D$ induced on the vertices of
$\D$ whose labels divide $\m$, and let $\LCM(I)$ denote the set of monomials that are least common multiples of arbitrary subsets of the minimial monomial generating set of $I$.

\begin{theorem}[\cite{CEFMMSS} {\bf Criterion for quasi-trees supporting resolutions}]\label{t:res-by-quasitree} Let $\D$ be a quasi-tree whose vertices are labeled
  with the monomial generating set of a monomial ideal $I$ in the
  polynomial ring $S$ over a field $\sfk$. Then $\D$ supports a
  resolution of $I$ if and only if for every monomial $\m$ in $\LCM(I)$, $\D_\m$ is empty or connected.
\end{theorem}

 If $I$ is minimally generated by $q$ monomials, then
  $I^2$ is minimally generated by at most ${q+1}\choose{2}$
  monomials. Our goal in this paper is to find a
  (smaller) subcomplex of the ${q+1}\choose{2}$-simplex which
  produces a free resolution of $I^2$, and only depends on $q$. The
  quasi-tree $\LL^2_q$, introduced in the next section, is such a
  candidate: it has exactly ${q+1}\choose{2}$ vertices, and for any
  given ideal $I$ with $q$ generators, it has an induced subcomplex
  $\LL^2(I)$ contained in $\Taylor(I^2)$ which supports a free
  resolution of $I^2$.

%%%%%%%%%%%%%%%%%%%%%%%%%%%%%%%%%%%%%%%%%%%%%%%%%%%%%%%%%%%%%%%%%
\section{The quasi-trees $\LL_q^2$ and $\LL^2(I)$}
\label{sec:main}
%%%%%%%%%%%%%%%%%%%%%%%%%%%%%%%%%%%%%%%%%%%%%%%%%%%%%%%%%%%%%%%%%
    
For an integer $q \geq 1$ we now give a description of a simplicial
complex $\LL^2_q$, a subcomplex of the ${q+1}\choose{2}$-simplex .  We will show that if $I$ is a monomial ideal
generated by $q$ square-free monomials, an induced
subcomplex of $\LL^2_q$, which we denote by $\LL^2(I)$, always
supports a free resolution of $I^2$. The complex $\LL^2_q$ is a far
smaller subcomplex of the ${q+1}\choose{2}$-simplex, and its construction is motivated by the monomial orderings used to build the Lyubeznik complex \cite{L}.

\begin{definition}\label{d:L2}  For an integer $q \geq 3$, the simplicial complex $\LL^2_q$ over the vertex set
  $\{\ell_{i,j} \st 1 \leq i \leq j \leq q\}$ is defined by its facets as: 
$$\LL^2_q=\tuple{\{\ell_{i,j} \st 1 \leq j \leq q \}_{1 \leq i \leq q}, \quad
\{\ell_{i,j} \st 1 \leq i< j \leq q \}},$$
where we define $\ell_{j,i}$ for $j > i$ by the equality
$\ell_{j,i}=\ell_{i,j}$. 
For $q=1$ and $q=2$ we use the same construction but note that $\{\ell_{i,j} \st 1 \leq i < j \leq q\}$ is empty for $q=1$ and is a face but not a facet for $q=2$. 
\end{definition}

When $q = 1$, the ideals $I$ and $I^r$ for all $r \geq 2$  are principal and $\LL_q^2$ is a point.
When $q=2$, the complex $\LL_2^2$ has only $2$ facets, see
\cref{e:L2-picture}. Note that $\LL^2_q$ has ${q+1}\choose{2}$
vertices, which is the number of vertices of the
${q+1}\choose{2}$-simplex, and, when $q>2$, it has $q+1$ facets, where
one facet has dimension ${{q}\choose{2}} -1$ and the remaining $q$
facets have dimension $q-1$.

\begin{example}\label{e:L2-picture}
  The complexes $\LL^2_3$ and $\LL^2_2$ are shown on the left and
  right, respectively.

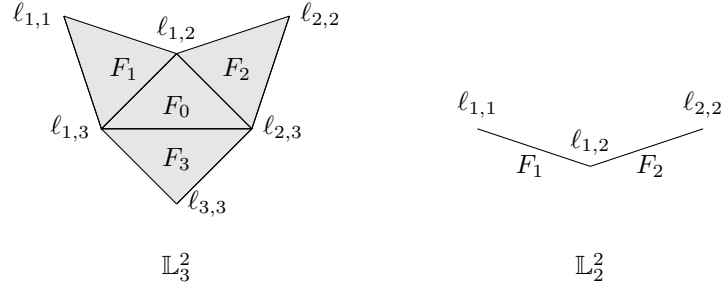
\[ 
\begin{tikzpicture}
\tikzstyle{point}=[inner sep=0pt]
\node (a)[point,label=above:$\ell_{1,2}$] at (0,1) {};
\node (b)[point,label=left:$\ell_{1,3}$] at (-1,0) {};
\node (c)[point,label=right:$\ell_{2,3}$] at (1,0) {};
\node (d)[point,label=left:$\ell_{1,1}$] at (-1.5,1.5) {};
\node (e)[point,label=right:$\ell_{2,2}$] at (1.5,1.5) {};
\node (f)[point,label=right:$\ell_{3,3}$] at (0,-1) {};
\node (g)[point,label=below:$\LL^2_3$] at (0,-1.5) {};
\draw [fill=gray!20](a.center) -- (b.center) -- (c.center);
\draw [fill=gray!20](a.center) -- (b.center) -- (d.center);
\draw [fill=gray!20](a.center) -- (c.center) -- (e.center);
\draw [fill=gray!20](b.center) -- (c.center) -- (f.center);
\node (h)[point,label=above: $F_0$] at (0,0) {};
\node (h)[point,label=above: $F_1$] at (-0.7,0.5) {};
\node (h)[point,label=above: $F_2$] at (0.8,0.5) {};
\node (h)[point,label=above: $F_3$] at (0,-0.7) {};
\draw (a.center) -- (b.center);
\draw (a.center) -- (c.center);
\draw (a.center) -- (d.center);
\draw (a.center) -- (e.center);
\draw (b.center) -- (c.center);
\draw (b.center) -- (d.center);
\draw (b.center) -- (f.center);
\draw (c.center) -- (e.center);
\draw (c.center) -- (f.center);
\node (A)[point,label=above: $\ell_{1,2}$] at (5.5,-.5) {};
\node (B)[point,label=above: $\ell_{2,2}$] at (7,0) {};
\node (C)[point,label=above: $\ell_{1,1}$] at (4,0) {};
\node (G)[point,label=below:$\LL^2_2$] at (5.5,-1.5) {};
\node (E)[point,label=below: $F_1$] at (4.7,-.2) {};
\node (F)[point,label=below: $F_2$] at (6.3,-.2) {};
\draw (A.center) -- (B.center);
\draw (A.center) -- (C.center);
\end{tikzpicture} \]
\end{example}

\begin{proposition}\label{p:L2-quasitree} For $q\geq 1$, $\LL^2_q$ is a
  quasi-tree.
\end{proposition}

\begin{proof} If $q = 1$, then $\LL_q^2$ is a simplex of dimension 0,
  and so is a quasi-tree.  If $q=2$,
  there are only two facets, namely $F_1$ and $F_2$ (as depicted above), and
  $F_2$ is a leaf of $\langle F_1, F_2\rangle$ with joint $F_1$, so
  $\LL^2_2$ is a quasi-tree.  For $q \geq 3$, order the facets of
  $\LL^2_q$ by $F_0=\{\ell_{i,j} \st 1 \leq i< j \leq q \}$, and
  $F_i=\{\ell_{i,j} \st 1 \leq j \leq q \}$ for $1 \leq i \leq q$. By
  definition, if $i \neq k$ are nonzero, then
  $F_i \cap F_k = \{\ell_{i,k}\}\subseteq F_0.$ Thus each $F_i$ is a
  leaf of $\tuple{F_0,\ldots,F_i}$ with joint $F_0$, and we are done.
  \end{proof}

Given a square-free monomial ideal $I$, we now define a labeled induced subcomplex of $\LL_q^2$, denoted $\LL^2(I)$, which is obtained by deleting vertices from $\LL_q^2$. 

\begin{definition}[{\bf $\LL^2(I)$}]\label{d:L2I}
 For an  ideal $I$ minimally generated by the square-free monomials
$m_1,\ldots,m_q$, we define $\LL^2(I)$ to be a labeled induced subcomplex of $\LL_q^2$ formed by the following rules: 

\begin{enumerate}

\item Label each vertex of $\ell_{i,j}$ of $\LL^2_q$ with the monomial $m_im_j$.

\item If for any indices $i,j,u,v \in [q]$ where $[q]=\{1, \ldots, q\}$ with $\{i,j\}\ne \{u,v\}$ we have $m_im_j \mid m_um_v$, then 
\begin{itemize} 
	\item If $m_im_j = m_um_v$ and $i=\min \{i,j,u,v\}$, then delete the vertex
  $\ell_{i,j}$.

\item If $m_im_j \not= m_um_v$, then
  delete the vertex $\ell_{u,v}$.
  \end{itemize}
  
  \item Label each of the remaining faces with the least common
    multiple of the labels of its vertices.

\end{enumerate}

The remaining labeled subcomplex of $\LL^2_q$ is called $\LL^2(I)$,
 and is a subcomplex of $\Taylor(I^2)$.

\end{definition}

\begin{remark}\label{r:L}

\item  It follows from \cref{p:basic-indices-new} below that if $m_i^2$ divides $m_um_v$ then $u=v=i$, hence the vertices $\ell_{i,i}$ are not deleted in the construction of $\LL^2(I)$. 

\item In Step~2 above, when there is equality, the choice was made to
  eliminate the vertex $\ell_{i,j}$ with minimum index $i$ so that one
  has a well-defined definition for $\LL^2(I)$; in fact, one could
  show that a different choice of elimination would also serve our
  purposes.
\end{remark}

\begin{example}\label{e:running} Let $I=(abe,bc,cdf,ad)$.
  Setting $m_1=abe$, $m_2=bc$ and $m_3=cdf$, $m_4=ad$, we first label
  all vertices of $\LL^2_4$ with the products $m_im_j$, but then note
  that $m_2m_4 \mid m_1m_3$.
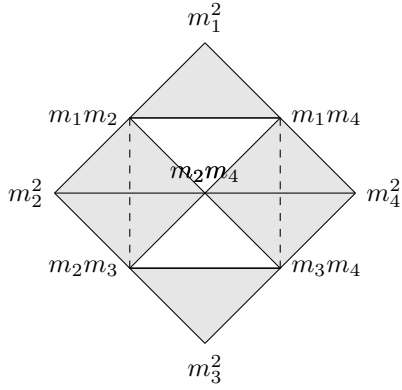
\[ 
\begin{tikzpicture}
\tikzstyle{point}=[inner sep=0pt]
\node (11)[point,label=above:$m_1^2$] at (0,2) {};
\node (12)[point,label=left:$m_1m_2$] at (-1,1) {};
\node (14)[point,label=right:$m_1m_4$] at (1,1) {};
\node (22)[point,label=left:$m_2^2$] at (-2,0) {};
\node (23)[point,label=left:$m_2m_3$] at (-1,-1) {};
\node (24)[point,label=above:$m_2m_4$] at (0,0) {};
\node (33)[point,label=below:$m_3^2$] at (0,-2) {};
\node (34)[point,label=right:$m_3m_4$] at (1,-1){};
\node (44)[point,label=right:$m_4^2$] at (2,0){};
\draw [fill=gray!20](11.center) -- (12.center) -- (14.center);
\draw [fill=gray!20](22.center) -- (12.center) -- (24.center);
\draw [fill=gray!20](22.center) -- (23.center) -- (24.center);
\draw [fill=gray!20](23.center) -- (23.center) -- (34.center);
\draw [fill=gray!20](33.center) -- (34.center) -- (23.center);
\draw [fill=gray!20](44.center) -- (34.center) -- (24.center);
\draw [fill=gray!20](44.center) -- (14.center) -- (24.center);
\draw [dashed] (12.center) -- (23.center);
\draw [dashed] (14.center) -- (34.center);
\draw (12.center) -- (14.center);
\draw (22.center) -- (24.center);
\draw (23.center) -- (34.center);
\draw (44.center) -- (24.center);
\draw (33.center) -- (23.center);
\draw (11.center) -- (14.center);
\node (24)[point,label=above:$m_2m_4$] at (0,0) {};
\end{tikzpicture} \]

So the (labeled) facets of $\LL^2(I)$ are the following five:
$$\begin{array}{c|c}
    \mbox{Facet} & \mbox{Dimension}\\
    \hline
          \{ m_1^2, m_1m_2,m_1m_4\}&2\\
\{ m_2^2, m_1m_2,m_2m_3,m_2m_4\}&3\\
\{ m_3^2, m_2m_3,m_3m_4\}&2\\
\{ m_4^2, m_1m_4,m_2m_4,m_3m_4\}&3\\
    \{  m_1m_2,m_1m_4,m_2m_3,m_2m_4,m_3m_4\}& 4\\ \hline
    \end{array}$$

    In particular, $\LL^2(I)$ is a $4$-dimensional complex labeled
    with the generators of $I^2$.
\end{example}  

We now present two preliminary results needed for the proof that when the ideal $I$ is square-free, $\LL^2(I)$ supports a
free resolution of $I^2$.

\begin{proposition}\label{p:basic-indices-new} Let $m_1,\ldots,m_q$ be a
  minimal square-free monomial generating set for an ideal $I$, let $r$ be a
  positive integer, and suppose that for some $i
  \in [q]$ and $1 \leq u_1 \leq \cdots \leq u_r \leq q$,
  $$m_i^r \mid m_{u_1} \cdots
  m_{u_r} \quad \mbox{ or } \quad m_{u_1} \cdots m_{u_r} \mid m_i^r.$$ 
   Then $u_1=\cdots=u_r=i$.
\end{proposition}  

\begin{proof} If for all, or some, of $j \in [r]$ we have $u_j=i$, then
  those copies of $m_i$ can be deleted from each side of the
  division, so one can assume, without loss of generality that $i=1 <
  u_1 \leq \cdots \leq u_r \leq q$.  Suppose that 
  $$m_1=x_1^{a_1}\cdots x_n^{a_n}
  \quad \mbox{ and } \quad  
  m_{u_j}=x_1^{b^j_1}\cdots x_n^{b^j_n},$$
  where
  $a_v,b^j_v \in \{0,1\}$ for $j \in [r]$ and $v \in [n]$. It follows that:

  \begin{itemize}
    \item if $m_i^r \mid m_{u_1} \cdots m_{u_r}$, then for every index
      $v\in [n]$ where $a_v\neq 0$, we have $ra_v=r$ and so
      $b^1_v=\cdots=b^r_v=1$. Therefore, we have $m_1 \mid m_{u_j}$ for
      $j \in[r]$.  This is a contradiction since these monomials are
      minimal generators of $I$.

    \item if $m_{u_1} \cdots m_{u_r} \mid m_i^r$, then for each
      nonzero exponent $b^1_v$ of $m_{u_1}$ we must have $a_v\neq 0$, and
      so $m_{u_1} \mid m_1$, again a contradiction.
  \end{itemize} 
\end{proof}

\begin{proposition}\label{p:irredundant-gens-I2} Let $I$ be an ideal minimally
  generated by square-free monomials $m_1,\ldots,m_q$ with $q \geq
  2$.  Then for every $i \in [q]$ there is a $j \in [q]\setminus
  \{i\}$ such that
  $$m_um_v \nmid m_im_j \mbox{ for any choice of }u,v \in [q]\setminus
  \{i,j\}.$$ In particular, $m_im_j$ is a minimal generator of $I^2$.
\end{proposition}

\begin{proof}
  Suppose, by way of contradiction, that there exists $i \in [q]$ such
  that for every $j \in [q]\setminus \{i\}$ there exist $u,v \in
  [q]\setminus \{i,j\}$ such that $m_um_v \mid m_im_j$.
 
 With $i$ as above, there exist functions  $\varphi, \psi\colon [q]\setminus \{i\}\to  [q]\setminus \{i\}$ such that 
  \begin{equation}
  \label{psi-phi-divides}
  m_{\varphi(j)}m_{\psi(j)}\mid m_im_j\quad\text{for all $j \in [q]\setminus \{i\}$.}
 \end{equation}
 For each $k\ge 0$, let $\varphi^k$ denote the composition
 $\varphi\circ \varphi\circ\dots\circ \varphi$ ($k$ times). (When
 $k=0$, $\varphi^0$ is the identity function.) Let
 $a\in [q]\setminus \{i\}$.  For each $w\ge 1$, set
 $b_w=\psi(\varphi^{w-1}(a))$. Apply \eqref{psi-phi-divides} with
 $j=\varphi^{k-1}(a)$ to get:
 \[
  m_{\varphi^k(a)}m_{b_k}\mid m_im_{\varphi^{k-1}(a)}\quad\text{for all $k\ge 1$}\,.
 \]
 From this, it is easy to see that
\[
\left.\left(m_{\varphi^k(a)}\cdot \prod_{w=1}^{k}m_{b_w}\right) \right| \left(m_im_{\varphi^{k-1}(a)}\cdot \prod_{w=1}^{k-1}m_{b_w} \right)\quad\text{for all $k\ge 2$\,.}
 \]

Inductively, we thus obtain
\begin{equation}
\label{s-k}
\left.\left(m_{\varphi^k(a)}\cdot \prod_{w=1}^{k}m_{b_w} \right) \right| \left(m_i^s m_{\varphi^{k-s}(a)}\cdot \prod_{w=1}^{k-s}m_{b_w} \right)\quad \text{for all $k\ge 2$ and $k > s\ge 1$.}
\end{equation}

Assume $\varphi^k(a)= \varphi^{k-s}(a)$ for some $k\ge 2$ and some $s$ with $k > s \ge  1$. After simplifying in \eqref{s-k} we obtain 
$$
\left.\left(\prod_{w=k-s+1}^{k}m_{b_w} \right) \right| m_i^s. 
$$
For $s=1$, this implies $m_{b_k} \mid m_i$, but since $b_k \neq i$, this contradicts the minimality of the generating set. If $s > 1$ this is a contradiction according to \cref{p:basic-indices-new}. Therefore, we have shown that the integers $\varphi(a), \varphi^2(a), \dots$ are distinct. This is a contradiction, since $\varphi^k(a)\in [q]\smallsetminus\{i\}$ for all $k$, and $[q]\smallsetminus\{i\}$ is a finite set. 
\end{proof}

We are now ready to prove the main result of the paper. 
  
\begin{theorem}[{\bf Main Result}]\label{t:L2-supports-new}
  Let $I$ be a square-free monomial ideal. Then $\LL^2(I)$ supports a
  free resolution of $I^2$.
\end{theorem}

\begin{proof}
  Suppose $I$ is minimally generated by the square-free monomials $m_1,\ldots,m_q$.

  The simplicial complex $\LL^2(I)$ is an induced subcomplex of the
  quasi-tree $\LL^2_q$ (\cref{p:L2-quasitree}), and is therefore 
  a quasi-forest itself (see~\cite{CEFMMSS,FH}).  Let $V$ denote the set of vertices of $\LL^2(I)$.  In view of
  \cref{t:res-by-quasitree}, to show that $\LL^2(I)$ supports a
  resolution of $I^2$, we need to show that, for every $\m \in \LCM(I^2)$,
  $\LL^2(I)_\m$ is connected, where $\LL^2(I)_\m$ is the
   induced subcomplex of the complex $\LL^2(I)$ on the set 
   $V_\m=\{\ell_{i,j}\in V
   \st m_im_j \mid \m\}$.

   Suppose $\m \in \LCM(I^2)$. If $q=1$, then
     $\LL^2(I)_\m$ is either empty or a point. If $q=2$, then
     $I^2=(m_1^2, m_1m_2,m_2^2)$ and $\LL^2(I)$, as pictured in
     \cref{e:L2-picture}, has two facets connected by the vertex
     $\ell_{1,2}$. If $\m \in \{m_1^2, m_2^2\}$, then $\LL^2(I)_\m$ is
     a point, and hence connected. Otherwise, $m_1m_2 \mid \m$, so the
     vertext $\ell_{1,2}$ will be in $\LL^2(I)_\m$. If either
     $\ell_{1,1}$ or $\ell_{2,2}$ are in $\LL^2(I)_\m$, they will be
     connected to $\ell_{1,2}$. Therefore $\LL^2(I)_\m$ is
     connected.

   Now assuming $q \geq3$, we use the notation introduced in the proof
   of \cref{p:L2-quasitree} for the facets of $\LL^2_q$, namely
   $F_0, \dots, F_q$. The facets of $\LL^2(I)_\m$ are the maximal sets
   among the sets $F_0\cap V_\m, \dots, F_q\cap V_\m$.
   
 If $\m=m_i^2$ for some $i\in [q]$, then \cref{p:basic-indices-new}
 shows that $\LL^2(I)_\m$ is one point, and hence is connected. Assume
 now that $\m\ne m_i^2$ for all $i\in [q]$, and hence
 $F_0\cap V_\m\ne \varnothing$. To show that $\LL^2(I)_{\m}$ is
 connected, it suffices to show that, for each $i\in [q]$ such that
 $F_i\cap V_\m\ne \varnothing$, the intersection between
 $F_i\cap V_\m$ and $F_0\cap V_\m$ is nonempty. Note that any vertex in
 $F_i\cap V_\m$ other than $\ell_{i,i}$ is also in $F_0\cap V_\m$.  We
 thus need to show that if $\ell_{i,i}\in V_\m$ for some $i\in [q]$,
 then there exists $b\in [q]$ with $b\ne i$ such that
 $\ell_{i,b}\in V_\m$.
   
Assume  $\ell_{i,i}\in V_\m$, hence $m_i^2\mid \m$.
 Set
$$
A=\{j\in [q]\colon m_j\mid \m\}\,.
$$
Note that $i\in A$. Since $\m\ne m_i^2$, we see that $|A|\ge 2$. By~\cref{p:irredundant-gens-I2}
    applied to the ideal generated by the monomials $m_j$ with $j\in A$, there exists $b \in A\smallsetminus \{i\}$ such that 
    \begin{equation}
    \label{does-not-divide}
    m_um_v \mbox{ does not divide } m_im_b \mbox{ for all }
    u,v \in A\setminus\{i,b\}\,.
    \end{equation}
    
    Since $b\in A$, we have $m_b\mid \m$. We claim that $m_im_b\mid \m$ as well. Indeed, since $m_b$ is a
    square-free monomial, setting $\m=m_i^2\n$, one has
  \begin{equation}\label{e:mj}
    m_b\mid \m \Rightarrow m_b\mid m_i^2\n \Rightarrow m_b \mid
  m_i\n \Rightarrow m_im_b \mid m_i^2\n \Rightarrow m_im_b \mid
  \m\,.
  \end{equation}
  
 In order to conclude $\ell_{i,b}\in V_\m$, we need to show that  $\ell_{i,b}\in V$, that is, $\ell_{i,b}$ is a vertex of $\LL^2(I)$.  
 If $\ell_{i,b} \notin V$, then we must have $m_um_v \mid m_im_b$ for some $u,v
    \in [q]\smallsetminus \{i,b\}$.  Since $m_im_b\mid \m$, we further have $m_u\mid \m$ and $m_v\mid \m$, hence $u,v\in A$. 
    This contradicts \eqref{does-not-divide} above. 
    \end{proof}
    
    \begin{remark} \label{minimal} Given any $q \geq 2$, there are square free monomial ideals $I$ with $q$ generators such that $\LL^2(I)=\LL^2_q$ and the resolution supported on $\LL^2(I)$ is minimal. The ideal $I =(xabc, yade, zbdf, wcef)$ is such an example when $q=4$, (see \cite{CEFMMSS}). 
    \end{remark}

\section{A bound on the Betti numbers of $I^2$}\label{s:bounds}

We now consider bounds on the Betti numbers of the second power of a
  square-free monomial ideal $I$, as provided by the simplicial complex
  $\LL^2(I)$.  Since $I^2$ has a free resolution supported
  on $\LL^2(I)$, $\beta_d(I^2)$ is bounded above by the number of
  $d$-faces of $\LL^2(I)$, which itself is bounded above by the number
  of $d$-faces of $\LL^2_q$.

  It can be seen from the proof of \cref{c:betti-bound} below that the
  right-hand term of the inequality $(a)$ below is precisely the
  number of $d$-faces of $\LL^2_q$. Note that the bound in $(a)$ depends only on the number of
  generators $q$, and not on $I$ itself.  The right-hand term of
  the inequality $(b)$ below is equal to the number of $d$-dimensional
  faces of $\LL^2(I)$, which provides a more precise bound that is
  dependent on the ideal $I$.

\begin{theorem}
\label{c:betti-bound}
Let $I$ be a square-free monomial ideal minimally generated by $q\ge 2$ monomials. Then for each $d\ge 0$ 
the $d^{th}$ Betti number $\beta_d(I^2)$ satisfies 
\[ \begin{aligned}
(a) \hspace{.3in}  \beta_d(I^2) &\le \binom{{\frac{1}{2}(q^2-q)}}{d+1}+q {{q-1}\choose{d}}.
\end{aligned}\]
Furthermore, setting $s$ to be the minimal number of generators of $I^2$ and $t_i$ to be the number of vertices of the form $\ell_{i,j}$ that were deleted from $\LL_q^2$ when forming $\LL^2(I)$, then
\[ \begin{aligned}
 (b) \hspace{.3in} \beta_d(I^2) &\le {{s-q}\choose{d+1}}+\sum_{i=1}^q {{q-1-t_i}\choose{d}}.
\end{aligned}\]
\end{theorem}

By \cref{minimal}, the bound in $(a)$ is sharp.

\begin{proof} We begin by proving inequality $(b)$. \cref{t:L2-supports-new} gives that for each $d\ge 0$,   $\beta_d(I^2)$ is bounded above by the
  number of $d$-dimensional faces of $\LL^2(I)$.  We compute this number next. 

The faces of $\LL^2(I)$ are of two types:
\begin{enumerate}
\item Faces that do not contain any vertex of the form $\ell_{i,i}$ for $i\in [q]$. 
\item Faces that contain a vertex $\ell_{i,i}$ for some $i\in [q]$, and, as a consequence, all the other vertices have the form $\ell_{i,j}$ with $j\in [q]\smallsetminus\{i\}$. 
\end{enumerate}

Let $s$ denote the minimal number of generators of $I^2$ and set $t={{q+1}\choose{2}}-s$. Since ${{q+1}\choose{2}}$ is the number of vertices of $\LL^2_q$, the integer $t$ is precisely the number of vertices that are deleted in the construction of $\LL^2(I)$, as described in \cref{d:L2I}. As noted in \cref{r:L}, all the deleted vertices $\ell_{i,j}$ must satisfy $i\ne j$, hence the number of vertices $\ell_{i,j}$ of $\LL^2(I)$ with $i,j\in [q]$ and $i\ne j$ is  ${q\choose 2}-t$, which is equal to $s-q$.  

To construct a $d$-dimensional face of type (1), we need to choose $d+1$ vertices among the vertices $\ell_{i,j}$ of $\LL^2(I)$ with $i,j\in [q]$ and $i\ne j$. As noted above, there are $s-q$ such vertices.  Thus, the number of $d$-dimensional faces of type (1) is ${{s-q}\choose{d+1}}$. 

 Fix $i\in [q]$. To construct a $d$-dimensional face of type (2) that contains $\ell_{i,i}$,  we need to choose $d$ vertices among the vertices $\ell_{i,j}$ of $\LL^2(I)$ that satisfy $j\ne i$.  There are $q-1-t_i$ such vertices, where $t_i$ denotes the number of vertices $\ell_{i,j}$ of $\LL^2_q$ that are deleted in $\LL^2(I)$. Thus the number of $d$-dimensional faces of type (2)  is 
$\sum_{i=1}^q {{q-1-t_i}\choose{d}}$.

Putting the two computations above together, we have that the number of $d$-dimensional faces of $\LL^2(I)$ is equal to 
${{s-q}\choose{d+1}}+\sum_{i=1}^q {{q-1-t_i}\choose{d}}$, yielding the inequality $(b)$. 

Note that inequality $(a)$ follows from $(b)$ by setting  $t_i=0$ for all $i$ and $s= {{q+1}\choose{2}}$. In view of our computation above, the right-hand side of inequality $(a)$ is precisely the number of $d$-dimensional faces of $\LL^2_q$. 
\end{proof}

For comparison, the fact that $\Taylor(I^2)$
supports a free resolution of $I^2$ gives an inequality 
$$
\beta_d(I^2)\le {{\frac{1}{2}(q^2+q)}\choose{d+1}},
$$ where the binomial on the right side denotes the
  number of $d$-faces of a ${\frac{1}{2}(q^2+q)}$-simplex, which is
  the largest possible size for $\Taylor(I^2)$.

To get an idea how much \cref{c:betti-bound} improves on this bound, 
we present the following table, for $q=4$: 
$$
\renewcommand{\arraystretch}{0.5}
\begin{array}{c|c|c|c|c|c|c|c}
d&0&1&2&3&4&5&6\\
\hline
&&&&&&&\\
d\mbox{-faces of largest possible } \Taylor(I^2)&&&&&&&\\
&10&45&120&210&252&210&120\\
{{ 10 \choose{d+1}}}&&&&&&&\\
&&&&&&&\\
\hline
&&&&&&&\\
d\mbox{-faces of } \LL^2_q &&&&&&&\\
&10& 27 &32&19&6&1&0\\
{{6\choose{d+1}}+4{{3}\choose{d}}}&&&&&&\\
\end{array}
$$

\bigskip

  To put this in context, we examine two specific
  ideals with $4$ generators, and use Macaulay2 to  find the Betti numbers of these ideals.

\begin{example} 
 For the ideal $J=(x,y,z,w)$ Macaulay2   gives the following Betti table for $J^2$:
    $$\begin{array}{c|c|c|c|c}
        d&0&1&2&3\\
        \hline
        \beta_{d}(J^2)&10 &20 &15& 4
      \end{array}
    $$
 These Betti numbers should be compared with the bounds in the table above. 
 
 Now let $I=(abe,bc,cdf,ad)$ be the ideal \cref{e:running}.  The Betti numbers of $I^2$ as calculated by Macaulay2 are
    the following.
    $$
    \begin{array}{c|c|c|c|c}
        d&0&1&2&3\\
        \hline
        \beta_{d}(I^2)&9&14&6&0
      \end{array}
   $$
In this case we should compare these Betti numbers with the  bounds given by the Taylor complex with $9$ vertices and the  bounds given by the \cref{c:betti-bound}$(b)$. 
For the given ideal, we saw that $\LL^2(I)$ has 9 vertices, and $m_2m_4$ is an eliminated vertex, hence $s=9$,  $t_2=t_4=1$ and $t_1=t_3=0$ in \cref{c:betti-bound}$(b)$. We have: 
 
$$
\renewcommand{\arraystretch}{0.5}
\begin{array}{c|c|c|c|c|c|c|c}
d&0&1&2&3&4&5&6\\
\hline
&&&&&&&\\
d\mbox{-faces of } \Taylor(I^2)&&&&&&&\\
&9&36&84&126&126&84&36\\
{{ 9 \choose{d+1}}}&&&&&&&\\
&&&&&&&\\
\hline
&&&&&&&\\
d\mbox{-faces of } \LL^2(I) &&&&&&&\\
&9& 20 & 18 &7 &1&0&0\\
{{5\choose{d+1}}+2{{3}\choose{d}}+2{{2}\choose{d}}}&&&&&&\\
\end{array}
$$
 \end{example}
 
\bigskip

\subsubsection*{Acknowledgements} 
The bulk of this work was done during the 2019 Banff workshop ``Women
in Commutative Algebra''. We are grateful to the organizers, the
funding agencies (NSF DMS-1934391), and to the Banff International
Research Station for their hospitality.

Author \c Sega and Spiroff were partially supported by grants from the
Simons Foundation (\#354594, \#584932, respectively), and authors
Cooper and Faridi were supported by the Natural Sciences and
Engineering Research Council of Canada (NSERC).

%%%%%%%%%%%%%%%%%%%%%%%%%%%%%%%%%%%%%%%%%%%%%%%%%%%%%%%%%%%%%%%%%

\bibliography{bibliography}

\end{document}